\documentclass[10pt]{amsart}
\usepackage{amsmath,amsthm}
\usepackage[colorlinks=true,linkcolor=blue,urlcolor=blue,citecolor=blue]{hyperref}
\usepackage[all]{xy}
\usepackage{amsmath}
\usepackage{graphicx} 
\usepackage{amsfonts} 
\usepackage{amssymb}
\usepackage{color}
\usepackage{epsfig}

\usepackage{amsmath}
\usepackage{graphicx} 
\usepackage{amsfonts} 
\usepackage{amssymb}
\usepackage{color}
\usepackage{epsfig}

\newcommand{\R}{\mathbb{R}}

\newcommand{\io}{\displaystyle\int_{\Omega}}

\newcommand{\N}{\mathbb{N}}

\newcommand{\weakly}{\rightharpoonup} 
\newcommand{\mean}{-\!\!\!\!\!\!\!\int}

\newcommand{\wsc}{\stackrel{*}{\rightharpoonup}}

\newcommand{\nat}{{\Bbb N}}

\newcommand{\U}[1]{u^n_{#1}} 
\newcommand{\Dn}[1]{D^n_{#1}}

\newcommand{\osd}[1]{\overset{#1}{\cd}}
\newcommand{\os}[1]{\overset{#1}{\lra}}

\newcommand{\lra}{\longrightarrow}

\newcommand{\normm}[1]{\| #1 \|}

\newcommand{ \calG}{\mathcal{G}}
\newcommand{ \calF}{\mathcal{F}}
\newcommand{ \G}{\stackrel{G}{\longrightarrow}}

\newcommand{\ho}{H^1_0(\Omega;\R^N)}
\newcommand{\ol}{\overline}

\newcommand{\eq}{\begin{equation}}
\newcommand{\eneq}{\end{equation}}
\newcommand{\al}{\begin{align}}
\newcommand{\enal}{\end{align}}
\newcommand{\cd}{\rightharpoonup}

\begin{document}

	\newtheorem{theorem}{Theorem}[section]
	\newtheorem{lemma}[theorem]{Lemma}
	\newtheorem{proposition}[theorem]{Proposition}
	\newtheorem{corollary}[theorem]{Corollary}
	
	\theoremstyle{definition}
	\newtheorem{definition}[theorem]{Definition}
	\newtheorem{remark}[theorem]{Remark}
	\newtheorem{example}[theorem]{Example}
	\newtheorem{problem}[theorem]{Problem}
	
	\numberwithin{equation}{section}

	\renewcommand{\subjclassname}{%
		\textup{2010} Mathematics Subject Classification}


\title[Damage dynamics: A variational approach]{Damage dynamics: A variational approach}
\author[A. Garroni]{Adriana Garroni}
\author[C.J. Larsen]{Christopher J. Larsen}
\author[D. Sarrocco]{David Sarrocco}


\maketitle

\begin{abstract}
In this paper we construct, by means of a variational formulation, the solutions of a problem of elastodynamics which includes the effect of damage for the elastic material. The result is a wave equation with time dependent operators which represents the elastic coefficients of the material undergoing damage. The dynamics that we construct also satisfies a threshold condition with the same threshold value that characterizes the quasi-static evolution of damage (see \cite{GL}).
\end{abstract}

\maketitle 

\section{Introduction}

Material defects, such as fracture, plasticity, and damage, have been extensively studied using variational models. A common feature of these phenomena beyond elasticity is the presence of non-smooth solutions (which have discontinuities, topological singularities, etc..) for which classical PDE's methods are not available.  The main advantage of variational formulations is then the ease of showing existence of global minimizers, which can be used via incremental minimization in order to construct quasi-static evolutions. On the other hand evolution in many cases follows local minima rather than global.
Moreover in many applications inertia cannot be neglected and it is important to have a robust definition for the dynamics of defects. Many approaches are possible in order to deal with dynamics and define satisfactory evolutions (ranging from vanishing viscosity regularizations, and higher order regularizations, to dynamics of phase field models, see e.g. \cite{LRTT} in the case of damage, and \cite{MiRo} and the references therein for a general review). 

An important observation is that underlying these variational models is often a threshold criterion - fracture occurs where the stress is sufficiently large or has a sufficiently large singularity, plastic deformations occur where the stress reaches the yield surface, and materials undergo damage where the stress exceeds a threshold.  In particular threshold criteria are localized and should guide dynamics.
Nevertheless the correspondence between these variational approaches and the threshold criterion can be unclear. 

In this paper we consider a model for  brittle damage which was formulated  and analyzed in \cite{FM} and  subsequently refined in \cite{FG} and \cite{GL}. 
The model proposed by \cite{FM} is simple but within reach. It assumes that only two states are possible, undamaged and damaged. These states are given by two elastic well-ordered tensors, $A_s$ and $A_w$, and the energy of each displacement $u$ with damage region $D$ is given by
\begin{equation}\label{energygeneral}
\int_\Omega \chi_D A_we(u) + (1-\chi_{D})A_se(u) dx + k|D|- \int_\Omega f\,u\,dx,
\end{equation}
where $e(u)=\frac{\nabla u+ \nabla u^T}{2}$ is the symmetrized gradient of $u$. The energy accounts for the elastic energy stored in a damaged material and penalizes the volume of the damaged region. The constant $k$ can be viewed as the cost (per volume) for the material to undergo damage, or the energy dissipated (per volume) when the material undergoes damage. The function $f$ here represents a dead load.

In the particular case of anti-plane elasticity in dimension $2$, or more in general in the scalar case in dimension $N$, and assuming that $A_s$ and $A_w$ are isotropic, i.e, $A_s =\beta I$ and $A_w=\alpha I$, with $0<\alpha<\beta$, the energy is rewritten as
\begin{equation}\label{energy}
\int_\Omega \chi_D \alpha|\nabla u|^2 + (1-\chi_{D})\beta|\nabla u|^2 dx + k|D| - \int_\Omega f\,u\,dx.
\end{equation}

In \cite{GL}, for the latter case, the first two authors showed that  minimizers of this variational model (and the corresponding quasi-static evolution) satisfy a threshold condition, which states that there exists a value $\lambda$ which depends on the energy, and then  on $\alpha$, $\beta$ and $k$, such that, whenever the modulus of the strain $\nabla u$ exceeds the value $\lambda$ the material is damaged. In particular if $(u,D)$ is a pair minimizing the energy \eqref{energy} subject to boundary conditions or body forces, then one proves that $D=\{x\in \Omega:\ |\nabla u|>\lambda\}$.

  A natural question is then, if we use the variational formulation for damage, but in a dynamic setting, do we obtain solutions of the threshold problem?

  An additional complication is that minimizers might not exist, and the appropriate relaxed formulation needs to be studied instead.
  This can be easily seen by rewriting the energy as follows
  $$
  \int_\Omega W(\nabla u) \,dx - \int_\Omega f\,u\,dx\,,
  $$
  and 
  $$
  W(\varepsilon) =\min\left\{\frac{1}{2}\beta|\varepsilon|^2\,,\ \frac{1}{2}\alpha|\varepsilon|^2 +k\right \}\,.
  $$
  This energy density is not convex (and not quasi-convex in the general vector valued case of \eqref{energygeneral}, see e.g. \cite{AK}), thus in the minimization procedure we expect microstructure, which requires relaxation.
  The convex envelope of $W$ can be represented as follows
  $$
  W^{**}(\varepsilon)=\min_{\Theta\in [0,1]}\min_{A\in G_\Theta(\alpha,\beta)}\left\{\frac{1}{2}A\varepsilon\, \varepsilon +k\Theta\right\}\,,
  $$
  where $G_\Theta(\alpha,\beta)$ is the $G$-closure of $\alpha I$ and $\beta I$ mixed with volume fractions $1-\Theta$ and $\Theta$ (see \cite{FG} and below for more detail). 
  
  The relaxed quasi-static evolution constructed in \cite{FG} and then refined in \cite{GL}, in particular involves a rich class of damaged materials characterized by elastic tensors in the $G$-closure of the two states without any a priori assumption on their structure.
  
  Moreover in \cite{GL} it is also obtained a relaxed threshold condition, and showed that solutions, both static and quasi-static, of the relaxed variational problem satisfy the relaxed threshold condition.
  
  Then an important question is, since the variational problem needs to be relaxed, can we prove existence of solutions to the dynamic relaxed variational damage problem, and are they solutions to an appropriate relaxed threshold problem?  
  
  In this paper, we prove this existence, in
 the scalar case when the two states (damaged and undamaged) are isotropic, and we show that solutions are also relaxed threshold solutions. The condition that we prove in Theorem \ref{EA} involves that same value for the threshold that was found in the quasi-static case (where the threshold condition was also proved to be necessary). The question whether this value for the threshold is optimal also in the dynamic case is still open. This question is also related to the possibility of defining an alternative dynamics of damage based directly on the threshold condition (as a unique criterion in order to switch from the undamaged state to the damaged one) as discussed in \cite{GL}.
   
  Our strategy is to formulate a discrete in time variational approximation of the dynamical problem following the line used in \cite{DM-La} and \cite{LoS}  for the dynamic of fracture (see also \cite{DMDL} for the variational approach via time regularization proposed in \cite{ST1} and \cite{ST2}). We then  show convergence of the discrete dynamics in the spirit of the result of \cite{CCFM} for the homogenization of the wave equation with time dependent coefficients (see also \cite{To}).  
  We notice that in general there are issues concerning the well posedness of the wave equation with discontinuous coefficients  (see  \cite{CS2}, \cite{HS}, \cite{FM-92}, \cite{B-OFM},  \cite{Lu} and the references therein). 
  Here a key  property which guarantees well posedness is the monotonicity in time of the (possibly discontinuous) elastic coefficients, which is a consequence of the irreversibility of damage.

\section{The variational model and its relaxation}\label{sec-var-pb}

We fix $0<\alpha<\beta$ and a bounded domain $\Omega\subseteq\R^N$, $N\geq 1$. 
Given $f\in H^{-1}(\Omega)$, we consider the energy functional
\begin{equation}\label{eaeq}
E_0(u,D):=\frac{1}{2}\int_{\Omega}\sigma_D |\nabla u|^2dx +k|D| -\langle f,u\rangle
\end{equation}
where $\sigma_D=\alpha \chi_D +\beta(1-\chi_D)$ is the elastic coefficient of a material that undergone damage in the region $D\subseteq\Omega$. For definiteness we consider $u\in H^1_0(\Omega)$, i.e., zero boundary conditions for $u$. A minimizing sequence for this energy corresponds to a sequence of sets $D_n$, with the corresponding displacements $u_n$. 
As already noticed in the Introduction this energy may need to be relaxed and a minimizing pair $(u,D)$ may fail to exists.
The right notion in order to characterize the limits of minimizing sequences, and then the framework for the relaxed problem, is the
$G$-convergence of the elliptic operators associated to the coefficients $\sigma_{D_n}$
(see, e.g.,  \cite{DS}, or \cite{MuTa97} for the more general case of nonsymmetric linear operators and $H$-convergence).
We recall here the main features of this notion, specialized to the case under consideration.

Consider a sequence $A^n \in L^\infty(\Omega; \R^{N\times N})$,
we say that $A^n
\G A$, $A \in L^\infty(\Omega; \R^{N\times N})$,  if,  for every $f \in H^{-1}(\Omega;\R^N)$, the solutions $u^n$ of the equilibrium equations 
\[- \hbox{div}(A^n\nabla u^n) =f, \quad u^n \in \ho,\]
satisfy
\begin{equation}
\label{G-conv}
\left\{\begin{array}{l}u^{n} \rightharpoonup u, \mbox{ weakly in } \ho\\[4mm] A^{n}\nabla u^{n} \rightharpoonup A \nabla u, 
\mbox{ weakly in } L^2(\Omega;\R^{N\times N}),\end{array}\right.
\end{equation}
where $u$ is the solution of 
\[- \hbox{div}(A \nabla u) =f.
\]

Note that in the case of symmetric matrices (tensors for the vector valued case in general linearized elasticity) the first property of (\ref{G-conv}) is enough to characterize $G$-convergence and the second condition, which is in turn essential in the nonsymmetric case,  can be obtained as a consequence.

We denote
$$\calF(\alpha,\beta):= \{ B\in \R^{N\times N}, \mbox{ symmetric, such that }\alpha |\xi|^2\leq B\xi\xi \leq \beta|\xi|^2\}
$$
and list below
the main properties of $G$-convergence in the class $L^\infty(\Omega; \calF(\alpha,\beta))$:
\begin{enumerate}
	\item {\it Compactness}: for any sequence $A^n \in L^\infty(\Omega; \calF(\alpha,\beta))$, there exists a subsequence, $A^{k(n)}$, and
	$A\in  L^\infty(\Omega; \calF(\alpha,\beta))$ such that $A^{k(n)} \G A$;
	\item {\it Convergence of the energy}: if $A^n \G A$, then, with $u^n$ and $u$ defined as above,
	\[\io A^n\nabla u^n\nabla u^n\; dx \to \io A\nabla u\nabla u\; dx;\]
	\item  {\it Metrizability}: $G$-convergence is associated to a metrizable topology on $L^\infty(\Omega; \calF(\alpha,\beta))$;
	\item {\it Ordering}: if $B^n\leq A^{n}$ and $B^{n}\G B$, $A^{n} \G A$, then $B \leq A$
	(the inequalities are in the sense of quadratic forms);
	\item {\it Locality}: if $B^{n}\G B$, $A^{n} \G A$, and $\chi$ is a characteristic
	function on $\Omega$, then $\chi B^{n}+(1-\chi)A^n\G 
	\chi B+(1-\chi) A$;
	\item {\it Periodicity}: if $A^n(x):= A(nx)$, with $A \in L^\infty([0,1]^N;\calF(\alpha,\beta))$ periodic, then the whole
	sequence $A^{n}$ $G$-converges to $A^0$, which is the constant matrix given by
	\begin{equation*} 
	A^0 \xi\xi = \inf_{\varphi \mbox{ periodic }}\int_{[0,1]^N}A(y) (\xi+\nabla\varphi)(\xi+\nabla\varphi)\; dy\quad \forall \xi\in\R^N.
	\end{equation*}

\end{enumerate}

	In the case of a two-phase isotropic material, we consider $ A^n=\chi_D(nx) \alpha I+ (1-\chi_D(nx)) \beta I$ with $\chi_D$ a characteristic function of a set $D\subseteq[0,1]^N$ extended by periodicity to the whole of $\R^N$, and we speak of periodic mixtures with volume fraction  $\Theta:= |D|$ of material
$\alpha$. 

Here with a little abuse of notation and terminology we identify the matrix $A^n$ with
$$
\sigma_D(nx)=\chi_D(nx) \alpha + (1-\chi_D(nx)) \beta .
$$
The set 
of all $G$-limits resulting from the periodic mixture of $\alpha $ and $\beta $ with volume fractions
$\Theta$ and $1-\Theta$ is denoted by $G_\Theta(\alpha,\beta)$. 

The relevance of this set  is clarified by a famous unpublished result of localization due to Dal Maso and Kohn (see \cite{Ra} for the nonlinear case). It claims that the range of all possible mixtures of $\alpha$ and $\beta$ is given by periodic homogenization. More precisely if  $\Theta \in L^\infty(\Omega,[0,1])$ and we denote by $\calG_\Theta(\alpha,\beta)\subset
L^\infty(\Omega; \calF(\alpha,\beta))$ the set of all possible 
$G$-limits of $\sigma_{\chi^n}$, where $\chi^n\wsc\Theta$, then
\begin{equation}\label{g-gper}\calG_\Theta(\alpha,\beta)= \{ A \in L^\infty(\Omega; \calF(\alpha,\beta)): A(x) \in \ol G_{\Theta(x)}(\alpha,\beta), \mbox{ a.e. in }
\Omega\}.\end{equation}
The set of all possible mixtures of $\alpha$ and $\beta$, as the volume fraction varies from point to point, is the $G$-closure of $\alpha$ and $\beta$ and will be denoted by $\calG(\alpha,\beta)$.

One then shows that minimizing sequences $(u_n, D_n)$ of the energy $E_0(u,D)$  converge (up to a subsequence) to pairs $(u,A)$, with $u\in H^1_0(\Omega)$ and $A\in \calG(\alpha,\beta)$, in the sense that $u_n$ converges weakly to $u$ in $H^1(\Omega)$ and $\sigma_{D_{n}}$ $G$-converges to $A\in \calG_\Theta(\alpha,\beta)$, for some $\Theta\in L^\infty(\Omega,[0,1])$, which is the weak$^*$ limit of $\chi_{D_n}$, and the corresponding minimal energy is given by
$$
\frac{1}{2}\int_{\Omega}A\nabla u\nabla u dx +k\int_\Omega \Theta dx -\langle f,u\rangle.
$$

The latter considerations  illustrates what happen for one time step (the first one) of a variational model describing an evolution of damage. Now if we consider a time dependent body force $f\in W^{1,1}([0,T];H^{-1}(\Omega))$ the corresponding (unrelaxed) energy is given by
\begin{equation}\label{eaeq}
E(t,u,D):=\frac{1}{2}\int_{\Omega}\sigma_D |\nabla u|^2dx +k|D| -\langle f(t),u\rangle.
\end{equation}
A (relaxed) quasi-static evolution of this model (in which  inertia is neglected) is given by a family of triples $(u(t), A(t), \Theta(t))$ with  $A(t)\in \calG_{\Theta(t)}(\alpha,\beta)$, costructed in \cite{GL}, which satisfy a stability condition (ideally in a quasi-static evolution the configurations are local minima for the total energy at each time), an irreversibility property, and conservation of energy (see \cite{GL}, Definition 3 and Theorem 7).

The setting is now ready to include inertia and consider damage dynamics.
 In view of the above considerations the natural space in which we look for the relaxed damage dynamics is give by the set $\calG(\alpha,\beta)$.
Therefore, given the initial conditions $(u_0,v_0)\in H^1_0(\Omega)\times L^2(\Omega)$, respectively for the initial displacement and the corresponding velocity, and a suitable initial damage set $D_0\subset\Omega$ as described below, and a forcing term $f\in W^{1,1}([0,T];L^2(\Omega))$, for $T>0$,
 we will construct a time dependent operator $A(t,\cdot)\in \calG(\alpha,\beta)$ and a function $u\in W^{1,\infty}(0,T;L^2(\Omega))\cap L^{\infty}(0,T; H^1_0(\Omega))$,
satisfying, in the weak sense, the following problem
\begin{equation}
\begin{cases}
\ddot{u}(t,x) - {\rm div}(A(t,x)\nabla u(t,x))=f(t,x)\quad &\hbox{in} \ (0,T)\times\Omega\\
u(0,x)=u_0(x) &\hbox{in } \Omega\\
\dot{u}(0,x)=v_0(x) &\hbox{in } \Omega. 
\end{cases}
\end{equation}
The matrix $A(t,\cdot)$ then represents the elastic coefficient of the damaged material at time $t$ and the function $u$ the corresponding elasto-dynamic solution.

\section{The incremental problem and the main result}\label{sec-main}

In what follows for functions depending on time and space variables, $t$ and $x$, we will often write explicitly only the dependence on $t$. Moreover for the derivative with respect to $t$ we will use both notation $\dot{u}$ and $\partial_t u$.

We start by defining the incremental problems needed to construct our dynamics. 
For every $n\in\N$,  we fix a time scale $\Delta t=\tau_n$, with $\tau_n\to 0$,
and we consider a partition of the time interval $[0,T]$ given by points $t^n_0=0$ and $t^n_i$ with $i\geq 1$ such that $t^n_{i}-t^n_{i-1}=\Delta t$. To avoid heavy notation, sometimes we write $t_i$ instead of $t^n_i$.\\

Starting from initial conditions $D_0\subseteq\Omega$  and $(u_0,v_0)\in H^1_0(\Omega)\times L^2(\Omega)$ we define $(u^n_0,D^n_0):=(u_0,D_0)$ and iteratively $(u^n_i,D^n_i)\in H^1_0(\Omega)\times \mathcal{P}(\Omega)$  as follows:
\begin{itemize}
	\item  We first choose $(\tilde u_1^n, D_1^n)\in H^1_0(\Omega)\times \mathcal{P}(\Omega)$ such that
	\begin{equation}\label{incremental0}
	\begin{split}
	E(t^n_1,\tilde u_1^n, D_1^n)  &+\frac{1}{2} \left \| \frac{\tilde u_1^n-u_{0}}{\Delta t}- v_0   \right\|^{2}_{L^2}\\
	&\leq  \inf_{u\in H^1_0\,;\ D\supseteq D_{0}} E(t^n_1, u, D)  +\frac{1}{2} \left\| \frac{u-u_{0}}{\Delta t}- v_0   \right\|^{2}_{L^2} +\frac{\tau^2_n}{2 }.
	\end{split}
	\end{equation}
	Then, with fixed $D_1^n$, we define $u_1^n$ be the minimizer in $H^1_0(\Omega)$ of 
	\begin{equation}\label{mind0}
	E(t^n_1, u, D_1^n)  +\frac{1}{2} \left \| \frac{u-u_{0}}{\Delta t}- v_0   \right\|^{2}_{L^2}.
	\end{equation}
	\item Analogously for $i\geq 1$ we choose $(\tilde u_{i+1}^n, D_{i+1}^n)\in H^1_0(\Omega)\times \mathcal{P}(\Omega)$ such that
	\begin{equation}\label{incremental}
	\begin{split}
	& E(t^n_{i+1}, \tilde u_{i+1}^n, D_{i+1}^n) +\frac{1}{2} \left \| \frac{\tilde u_{i+1}^n-u_{i}^n}{\Delta t}- \frac{ u_{i}^n-u_{i-1}^n}{\Delta t}   \right\|^{2}_{L^2}\leq \cr & \inf_{u\in H^1_0\,; \ D\supseteq D_{i}^n}  E(t^n_{i+1}, u, D)  +\frac{1}{2} \left\| \frac{u-u_{i}^n}{\Delta t}- \frac{ u_{i}^n-u_{i-1}^n}{\Delta t}   \right\|^{2}_{L^2} +\frac{\tau^2_n}{2^i }
	\end{split}
	\end{equation}
	and define $u_{i+1}^n$ to be the minimizer in $H^1_0(\Omega)$ of 
	\begin{equation}\label{mind}
	E( t^n_{i+1}, u, D_{i+1}^n)  +\frac{1}{2} \left \| \frac{u-u_{i}^n}{\Delta t}- \frac{u_{i}^n-u_{i-1}^n}{\Delta t}   \right\|^{2}_{L^2}.
	\end{equation}
\end{itemize}

We define $D_n(0):=D_0$ and $u_n(0)=u_0$, and for every $t\in(t_i^n,t_{i+1}^n]$ we define the following piecewise constant  (and piecewise affine) functions
\begin{equation}\label{disc1}
\begin{split}
&\tilde u_n(t):=u_{i+1}^n
\qquad\quad
u_n(t):=u_i^n + (t-t_i^n)\frac{u_{i+1}^n-u_i^n}{\Delta t} \\
&D_n(t):=D_{i+1}^n \qquad\quad f_n(t):=f(t^{n}_{i+1}).
\end{split}
\end{equation}

Note that since $f\in W^{1,1}(0,T; L^2\Omega)$ then for all  $t\in(0,T)$
\begin{equation}\label{f-fn}
\left\|  f_n \right\|_{L^2(0,T; L^2(\Omega))}\leq C \left\|  f \right\|_{L^2(0,T; L^2(\Omega))}
\end{equation}
and 
$$
\|f_n(t)-f(t)\|_{L^2(\Omega)}=\|f(t_{i+1}^n)-f(t)\|_{L^2(\Omega)}\leq \int_t^{t_{i+1}^n}\|\dot{f}(s)\|_{L^2(\Omega)}ds,
$$
with $t\in(t_i^n,t_{i+1}^n]$. 
Therefore $f_n(t)$ converges strongly in $L^2(\Omega)$ to $f(t)$ and
  $\left\|  f_n(t) \right\|_{L^2(\Omega)}$ converges to $\left\|  f(t) \right\|_{L^2(\Omega)}$
 for all $t\in(0,T)$.

With the following theorem we show convergence of this discrete in time (energetic) dynamics to a continuous damage dynamics. 
We will assume that given $(u_0,v_0)\in H^1_0(\Omega)\times L^2(\Omega)$ the set $D_0$ satisfy
\begin{equation}\label{thrzero}
D_0\supseteq \{ |\nabla u_0|\geq\lambda\}
\end{equation}
with $\lambda=\sqrt{\frac{2\alpha k}{\beta(\beta-\alpha)}}$. 
This condition is consistent with the threshold property that we will prove in Theorem \ref{EA}. It remains open to show that, under suitable continuity properties for the data, the dynamic of damage that we construct is continuous at zero (i.e.,   $A(0)=\sigma_{D_0}$).

\begin{theorem}\label{EA}
	Let $(u_n(t), D_n(t))$ be the sequence of piecewise constant in time  evolution as in (\ref{disc1}) and assume that $D_0$ satisfies \eqref{thrzero}.
	
	There exists $u\in W^{1,\infty}(0,T;L^2(\Omega))\cap L^{\infty}(0,T; H^1_0(\Omega))$, $\theta(t)\in L^{\infty}(\Omega;[0,1])$ and $A(t)\in L^{\infty}(\Omega;\mathcal{F}(\alpha,\beta))$ such that, up to subsequences, for all $t\in (0,T]$
	\begin{equation}\label{nume}
	u_{n}(t)\osd{H^1}u(t), \text{\qquad} \chi_{D_{n}(t)}\osd{*}\theta(t), \text{\qquad} \sigma_{D_{n}(t)}\os{G}A(t)
	\end{equation}
	with $\theta(t)$ increasing and $A(t)$ decreasing in time. At $t=0$ we have
	\begin{itemize}
		\item[i)] $A(0^+)\leq \sigma_{D_0}$ in the sense of quadratic forms;
		\item[ii)] $u(0^+)=u_0$ a.e. in $\Omega$;
		\item[iii)] $\partial_t u(0^+)= v_0$ a.e. in $\Omega$.
	\end{itemize}
	Moreover $(u_n(t),D_n(t))$ and the limit $(u(t),\theta(t), A(t))$ satisfy the following properties:
	\begin{itemize}
		\item \textbf{Euler-Lagrange equation}: 
		it holds
		\begin{equation}\label{weak-equation-limit}
	-\int_0^T\int_\Omega\partial_t{u}\partial_t{\phi}\,dx\,dt 
		+\int_0^T\int_\Omega A \nabla  u\nabla\phi\,dx\,dt=\int_0^T\langle f(t),\phi(t)\rangle\,dt
		\end{equation}
		for every $\phi\in L^2(0,T;H_0^1(\Omega))\cap W_0^{1,1}(0,T;L^2(\Omega))$;
		\item \textbf{Energy  inequality}: given
		$$E_{tot}(t):=\frac{1}{2}\normm{\partial_t{u}(t)}_{L^2}^2+\frac{1}{2}\int_{\Omega}A(t)\nabla u(t) \nabla u(t) dx+k\int_{\Omega}\theta(t) dx-\langle f(t),u(t)\rangle,$$ 
		it holds
		\begin{equation}\label{enbal}
		E_{tot}(t)\leq E_{tot}(0^+)-\int_{0}^{t}\langle \partial_t{f}(s),u(s)\rangle ds,
		\end{equation}

		\item \textbf{Threshold condition}: for each $\delta>0$ it holds
		\begin{equation}\label{thrdisc}
		\lim_{n\lra\infty}|\{x\notin D_n(t) : |\nabla u_n(t)|>\lambda+\delta\}|=0
		\end{equation}
		with $\displaystyle{\lambda:=\sqrt{\frac{2\alpha k}{\beta(\beta-\alpha)}}}$.
	\end{itemize}
\end{theorem}

The rest of the paper is devoted to the proof of this result. The compactness results are consequences of Lemmas \ref{lemma1} and \ref{conv-G}, the Euler-Lagrange equation and the Energy  inequality will be derived in Section \ref{sec-var-pb}, while the Threshold condition in obtained in Section \ref{sec-threshold}.

\section{Apriori estimates and compactness}\label{sec-comp}

For every $n$ and $t\in(t_i^n,t_{i+1}^n]$ we define the following auxiliary function
\begin{equation}\label{venne}
v_n(t)=\frac{u_{i+1}^n-u_i^n}{\Delta t}+\frac{(t-t_{i+1}^n)}{\Delta t}\left(\frac{u_{i}^n-u_{i-1}^n}{\Delta t}- \frac{ u_{i+1}^n-u_i^n}{\Delta t}\right)\,.
\end{equation}

By the definition of the sequences $u_n$ and $v_n$ and the minimality of $u_n^i$ we deduce that for all $t\in(0,T)$ we have
\begin{equation}\label{E-L-discrete}
\int_\Omega\dot{v}_n(t)\phi\,dx +\int_\Omega\sigma_{D_n(t)} \nabla \tilde u_n(t)\nabla\phi\,dx=\langle f_n(t),\phi\rangle
\end{equation}
for every $\phi\in H^1_0(\Omega)$.

Arguing similarly to what is done in \cite{DM-La} here we deduce the apriori estimates that are needed to pass to the limit in the discrete scheme.

\begin{lemma}\label{lemma1}
	Let $u_n$, $\tilde u_n$, $D_n$, and $v_n$ be defined as in \eqref{disc1} and \eqref{venne}. Then there exists $C>0$ such that for all $t\in (t^n_{i},t^n_{i+1}]$
\begin{equation}\label{f-apriori}
\|\dot{u}_n(t)\|_{L^2}^2+\|\nabla u_n(t_n^{i+1})\|_{L^2}^2+\Delta t\int_0^{t_n^{i+1}}\!\!\!\!\!(\|\dot{v}_n(t)\|_{L^2}^2  +\|\nabla\dot{u}_n(t)\|_{L^2}^2) \,dt\leq C\,
\end{equation}
for each $i=0,\ldots ,\lfloor T/\Delta t \rfloor$ and all $n\in\nat$.
	Moreover we have that
	\begin{enumerate}
		\item[i)] $u_n$ is bounded in $W^{1,\infty}(0,T; L^2(\Omega))\cap L^{\infty}(0,T; H^1_0(\Omega))$;
		\item[ii)] $\tilde u_n$ is bounded in $L^\infty(0,T; H^1_0(\Omega))$;
		\item[iii)] $v_n$ is bounded in $W^{1,\infty}(0,T; L^2(\Omega))$.
			\end{enumerate}
		Finally there exists a subsequence of indices (still denoted by $n$) and a function $u\in H^2(0,T; L^2(\Omega))$ such that
		$$
u_n\to u \quad \hbox{in} \ H^1(0,T; L^2(\Omega))
$$
and
\begin{equation}\label{conv-punt}
 v_n(t)\weakly \dot{u}(t) \quad \hbox{in} \ L^2(\Omega),\ \hbox{for } \ t\in[0,T]
\end{equation}	
	with $u(0^+)=u_0$ and $\dot{u}(0^+)=v_0$.

\end{lemma}

\begin{proof}

	We start proving that
	for each $i=0,\ldots,\lfloor T/\Delta t \rfloor$ it holds
	\begin{equation}\label{ste}
	\begin{split}
	& \normm{\dot{u}_n(t)}_{L^2}^2+\int_{\Omega}\sigma_{D_n(t)}|\nabla{u_n(t^n_{i+1})}|^{2}dx+\Delta t\int_{0}^{t^n_{i+1}}\normm{\dot{v}_n(s)}_{L^2}^2ds\\
	&+\Delta t\int_{0}^{t^n_{i+1}}\int_{\Omega}\sigma_{D_n(t)}|\nabla \dot{u}_n(s)|^2dxds \\ 
	&= 2\int_{0}^{t^n_{i+1}}\langle f_{n}(s),\dot{u}_n(s)\rangle ds+\int_{\Omega}\sigma_{D_n(0)}|\nabla u_n(0)|^{2}dx+\normm{\frac{u_1^n-u_0}{\Delta t}}_{L^2}^{2}\\
	&-(\beta-\alpha)\sum_{j=0}^{i}\int_{D^n_{j+1}\backslash D^{n}_{j}}|\nabla u_n(t^{n}_{j})|^{2}dx
	\end{split}
	\end{equation}
	for $t\in (t^n_{i},t^n_{i+1}]$.
	
	Since $u^n_{j+1}$ is the minimum point for the functional in (\ref{mind}) it satisfies the following weak Euler-Lagrange equation
	\begin{equation}
	\int_{\Omega}\sigma_{D^n_{j+1}}\nabla u^n_{j+1}\nabla \varphi dx+\int_{\Omega}\bigg(\frac{u_{j+1}^n-2u_{j}^n + u_{j-1}^n}{\Delta t}\bigg)\frac{\varphi}{\Delta t}dx-\langle f(t^{n}_{j+1}),\varphi \rangle =0
	\end{equation}
	for each $\varphi\in H^{1}_0(\Omega)$. 
	Choosing $\varphi=u_{j+1}^n-u_{j}^n$ we have
	
	\begin{equation}
	\begin{split}
	& \int_{\Omega}\sigma_{D^n_{j+1}}|\nabla u^n_{j+1}|^{2}-\int_{\Omega}\sigma_{D^n_{j+1}}\nabla u^n_{j+1}\nabla u^n_{j}dx+\int_{\Omega}\bigg|\frac{u_{j+1}^n-u_{j}^n}{\Delta t}\bigg|^{2}dx \cr & -\int_{\Omega}\frac{u_{j+1}^n-u_{j}^n}{\Delta t}\frac{u_{j}^n-u_{j-1}^n}{\Delta t}dx-\langle f(t^{n}_{j+1}),(u_{{j+1}}^n-u_{j}^n) \rangle=0.
	\end{split}
	\end{equation}
	We then use the identity $\displaystyle{\frac{1}{2}\normm{g}_{L^2}^2-\int_{\Omega}g\cdot hdx=\frac{1}{2}\normm{g-h}^{2}_{L^2}-\frac{1}{2}\normm{h}^{2}_{L^2}}$ for the first two terms on the left hand side  with $\displaystyle{ g=\nabla u^{n}_{j+1}\sqrt{\sigma_{D^n_{j+1}}}}$ and $\displaystyle{h=\displaystyle{\nabla u^{n}_{j}\sqrt{\sigma_{D^n_{j+1}}}}}$, and $\displaystyle{g=\frac{u^n_{j+1}-u^n_j}{\Delta t}}$ and $\displaystyle{h=\frac{u^n_{j}-u^n_{j-1}}{\Delta t}}$ for the second two terms, and we obtain
	
	\begin{equation}\label{f-identity}
	\begin{split}
	& \left\|\frac{u^n_{j+1}-u^n_{j}}{\Delta t} \right\|_{L^2}^{2}+\left\|\frac{u^n_{j+1}-u^n_{j}}{\Delta t}-\frac{u^n_{j}-u^n_{j-1}}{\Delta t} \right\|_{L^2}^{2} \cr & +\int_{\Omega}\sigma_{D^n_{j+1}}|\nabla u^n_{j+1}|^{2}dx  +\int_{\Omega}\sigma_{D^n_{j+1}}|\nabla u_{{j+1}}^n-\nabla u^n_{j}|^2dx \cr &
	=2\langle f(t^{n}_{j+1}),u^n_{{j+1}}-u^n_{j}\rangle +\int_{\Omega}\sigma_{D^n_{j+1}}|\nabla u^{n}_{j}|^2 dx+\left\|\frac{u^{n}_{j}-u^{n}_{j-1}}{\Delta t}\right\|_{L^2}^2.
	\end{split}
	\end{equation}
	Summing over $j=0,..,i$  and using the identity $\sigma_{D_{j+1}^n}=\sigma_{D_{j}^n}-(\beta-\alpha)\chi_{D^n_{j+1}\backslash D^n_{j}}$  we have for $t\in (t^n_{i},t^n_{i+1}]$ \\ 
	\begin{equation*} 
	\begin{split}
	& \left\|\frac{u^n_{i+1}-u^n_i}{\Delta t} \right\|_{L^2}^{2}+\sum_{j=0}^i\normm{\Delta t \dot{v}_n(t^n_{j+1})}_{L^2}^{2}+  \int_{\Omega}\sigma_{D^n_{i+1}}|\nabla u^n_{i+1}|^{2}dx\\
	&+ \sum_{j=0}^i\int_{\Omega}\sigma_{D^n_{j+1}}|\nabla u_{{j+1}}^n-\nabla u^n_{j}|^2dx \cr 
	&
	=2\sum_{j=0}^i\langle f(t^{n}_{j+1}),(u^n_{j+1}-u^n_{j})\rangle +\int_{\Omega}\sigma_{D_n(0)}|\nabla u_n(0)|^2 dx+\left\|\frac{u^n_{1}-u_0}{\Delta t}\right\|_{L^2}^2\\
	&-(\beta-\alpha)\sum_{j=1}^{i}\int_{D^n_{j+1}\backslash D^n_{j}}|\nabla u^n_{j}|^{2}dx.
	\end{split}
	\end{equation*}
	From this it follows immediately (\ref{ste}), using the definitions in (\ref{disc1}), and  that  $\dot{u}_n(t)=\frac{u^n_{j+1}-u^n_j}{\Delta t}$. 
	Now from \eqref{ste} we deduce for each $i$
	\begin{equation}\label{f-mn}
	\begin{split}
&\sup_{t\in (t^n_{i},t^n_{i+1}]}		\|\dot{u}_n(t)\|_{L^2}^2+\alpha\|\nabla u_n(t^n_{i+1})\|_{L^2}^2+\Delta t\int_0^{t^n_{i+1}}\|\dot{v}_n(t)\|_{L^2}^2 dt\\ &\qquad +\alpha\Delta t \int_0^{t^n_{i+1}}\|\nabla\dot{u}_n(t)\|_{L^2}^2 dt \\ &\leq \beta\normm{\nabla u_0}_{L^2}+\normm{\frac{u^n_{1}-u_0}{\Delta t}}_{L^2}^{2}+2\normm{f_n}_{L^{2}(0,T;{L^2}(\Omega))}T^{1/2}\max_{t\in[0,T]}\normm{\dot{u}_n(t)}_{L^{2}}\\
&\leq \beta\normm{\nabla u_0}_{L^2}+\normm{\frac{u^n_{1}-u_0}{\Delta t}}_{L^2}^{2}+C\normm{f}^2_{L^{2}(0,T;{L^2}(\Omega))}T+ \frac12 \max_{t\in[0,T]}\normm{\dot{u}_n(t)}^2_{L^{2}}.
	\end{split}
	\end{equation}
Since (\ref{f-mn}) holds for every $i=0,\ldots, \lfloor T/\Delta t \rfloor$, we immediately have that 
	\begin{equation}\label{MN0}
\frac12\max_{t\in [0,T]}		\|\dot{u}_n(t)\|_{L^2}^2\leq\beta\normm{\nabla{u}_0}^2_{L^2(\Omega)}+\normm{\frac{u^n_{1}-u_0}{\Delta t}}^2_{L^2(\Omega)}+C\normm{f}^2_{L^{2}(0,T;{L^2}(\Omega))}T
	\end{equation}
Now by the definition of $u_1^n$ we have that
$$
k|D_1|+\frac{1}{2} \left \| \frac{u_1^n-u_{0}}{\Delta t}- v_0   \right\|^{2}_{L^2}\leq	E(t^n_1, u_0, D_0)  +\frac{1}{2} \left \|  v_0   \right\|^{2}_{L^2}
$$
and then
\begin{equation}\label{f-20-0}
\frac{1}{2} \left \| \frac{u_1^n-u_{0}}{\Delta t}- v_0   \right\|^{2}_{L^2}\leq	\beta \int_\Omega |\nabla u_0|^2dx +\langle f(t_1^n), u_0\rangle +\frac{1}{2} \left \|  v_0   \right\|^{2}_{L^2}\leq C
\end{equation}
which implies 
\begin{equation}\label{f-20}
\left \| \frac{u_1^n-u_{0}}{\Delta t}   \right\|^{2}_{L^2}\leq C
\end{equation}
Combining \eqref{f-20} and \eqref{f-mn}
we obtain \eqref{f-apriori}.

	From \eqref{f-apriori} we immediately obtain  (i) and (ii).
	
	Now from \eqref{f-identity} and using (i) and (ii) we obtain for every $t\in(t_i^n, t_{i+1}^n]$
	$$
\left\|\dot{v}_n(t) \right\|_{L^2}^{2}=\left\|\frac{u^n_{j+1}-u^n_{j}}{\Delta t}-\frac{u^n_{j}-u^n_{j-1}}{\Delta t} \right\|_{L^2}^{2}\leq C
	$$
	which gives (iii).

By the bounds (i), (ii), and (iii), we can then conclude that, up to a subsequence, $u_n\weakly u$ in $H^1(0,T; L^2(\Omega))$ and $v_n\weakly v$ in $H^1(0,T; L^2(\Omega))$. We also have that $\dot{u}(t)=v(t)$ in $L^2(\Omega)$ for all $t\in[0,T]$. Indeed if $t\in (t_{i}^n, t_{i+1}^n]$ it holds
\begin{equation}\label{est1}
	\normm{\dot{u}_n(t)-v_n(t)}_{L^2}=\normm{v_n(t^{n}_{i+1})-v_n(t)}_{L^2}\leq\int_{t^n_i}^{t^n_{i+1}}\normm{\dot{v}_n(s)}_{L^2}ds\leq c\tau_n,
\end{equation}
which goes to zero when $n\lra \infty$. As a consequence of \eqref{est1} we also obtain \eqref{conv-punt}.
Using the convergence of the $v_n$ we also deduce that $u_n$ converges strongly to $u$ in $H^1(0,T; L^2(\Omega))$ and that $u\in H^2(0,T; L^2(\Omega))$. From this, and the fact that $u_n(0)=u_0$, we easily deduce that $u(0^+)=u_0$.

It remains to show that $\dot{u}(0^+)=v_0$.  We first show that up to a subsequence
\begin{equation}\label{f-v0-1}
v_n(\tau_n)=\frac{u_1^n-u_{0}}{\tau_n}\weakly v_0 \quad \hbox{in } L^2(\Omega)
\end{equation}
(we recall that $\Delta t=\tau_n$).
Indeed  from \eqref{f-20-0} we deduce that $\frac{u_1^n-u_{0}}{\tau_n}$  converges weakly in $L^2(\Omega)$. The fact that its limit is $v_0$ is a consequence of the Euler Lagrange equation for $u_1^n$ which gives
$$
\left|\int_\Omega \bigl(\frac{u_1^n-u_{0}}{\tau_n}-v_0\bigr)\varphi dx\right|=\tau_n \left| \langle f(\tau_n),\varphi\rangle- \int_\Omega \sigma_{D_1^n}\nabla u_1^n\nabla \varphi dx \right|\leq C \tau_n \|\varphi\|_{H^1}
$$
for all $\varphi\in H^1_0(\Omega)$. 

Since $u\in H^2(0,T; L^2(\Omega))$ we have that $\dot{u}(t)\to \dot{u}(0^+)$ in $L^2(\Omega)$ as $t$ tends to $0$. Moreover, as in \eqref{est1}, we deduce that
\begin{equation}\label{f-v0-2}
\normm{v_n(\tau_n)-v_n(t)}_{L^2}\leq c|t-\tau_n|.
\end{equation}
This combined with \eqref{est1}, the convergence of the $u_n$, and \eqref{f-v0-1}, concludes that $\dot{u}(0^+)=v_0$ a.e. in $\Omega$.

\end{proof}

\begin{lemma}\label{conv-G}
There exists a subsequence, still indexed by $n$, such that for all $t\in(0,T]$ we have
\begin{equation}\label{conv-A-theta}
\chi_{D_{n}(t)}\osd{*}\theta(t), \text{\qquad} \sigma_{D_{n}(t)}\os{G}A(t)
\end{equation}
with $\theta(t)\in L^\infty(0,T, [0,1])$ increasing and $A(t)\in G_{\theta(t)}(\alpha,\beta)$ decreasing in time. Moreover $A(0^+)\leq \sigma_{D_0}$ in the sense of quadratic forms.
\end{lemma}

\begin{proof}
Compactness at each time $t$ is guaranteed by the compactness of the weak start topology in $L^\infty$ and the compactness of the $G$-convergence. The key point is the possibility of extracting a sequence that does not depend on $t$.
This a standard argument in this context of discrete in time approximation schemes, and it  is a direct consequence of the monotonicity of the sequences $\theta_n$ and $\sigma_{D_n}$ with respect to time, combined with the metrizability of the $G$-convergence and an application of the Helly's theorem for monotone sequence (see \cite{FG}, Theorem 2 and Remark 3).

The condition $A(0^+)\leq \sigma_{D_0}$ is a direct consequence of property (4), ordering, of the $G$-convergence (together with the monotonicity of $A(t)$ for the definition of $A(0^+)$).
\end{proof}

\begin{remark}
Note that by the approximate minimality of $D_n(t)$ we deduce a first threshold condition of the following form:
Given $M>\sqrt{\frac{2k}{\beta-\alpha}}$ we have
\begin{equation}\label{weak-threshold}
\liminf_{n\to+\infty}\int_{\{ |\nabla \tilde u_n(t)|>M\}\setminus D_n(t)}|\nabla \tilde u_n|^2\, dx=0.
\end{equation}
This, in particular, implies that
$$
\liminf_{n\to +\infty}|\{x\in \Omega\setminus D_n(t)\,:\ |\nabla \tilde u_n(t)|>M\}| =0\,.
$$
It easy to prove \eqref{weak-threshold} by contradiction. Assume that there exists $\delta>0$ such that
$$
\liminf_{n\to+\infty}\int_{\{ |\nabla \tilde u_n(t)|>M\}\setminus D_n(t)}|\nabla \tilde u_n|^2\, dx>\delta\,.
$$
We then add for every $n$ the set $E=\{ |\nabla \tilde u_n(t)|>M\}$ to the damage set $D_n(t)$ and we obtain a reduction of energy given by
$$
E(t_{i+1}^n, \tilde u_n(t),D_n(t))-E(t_{i+1}^n, \tilde u_n(t), D_n(t)\cup E)= \frac12(\beta-\alpha)\int_E |\nabla \tilde u_n(t)|^2dx -k|E|
$$
Now there are two possibilities: either $|\{ |\nabla \tilde u_n(t)|>M\}|\to 0$ or $|\{ |\nabla \tilde u_n(t)|>M\}|>\eta>0$ for $n$ large enough. In these two cases we get either
$$
\liminf_{n\to+\infty}E(t_{i+1}^n,\tilde u_n(t),D_n(t))-E(t_{i+1}^n, \tilde u_n(t), D_n(t)\cup E)\geq \frac12(\beta-\alpha)\delta>0
$$
or
$$
\liminf_{n\to+\infty}E(t_{i+1}^n, \tilde u_n(t),D_n(t))-E(t_{i+1}^n, \tilde u_n(t), D_n(t)\cup E)\geq[\frac12(\beta-\alpha)M^2 -k]\eta>0
$$
and both contradict the minimality \eqref{incremental}.

Property \eqref{weak-threshold} will be used in the derivation of the energy inequality below.
Note that this is not yet the threshold conditions of Theorem \ref{EA} since the constant $\sqrt{\frac{2k}{\beta-\alpha}}$ is higher than $\lambda$.
\end{remark}

\section{The characterization of the limit problem and the energy inequality}\label{sec-EL}

We prove in this section that the limit $u$ of the discrete in time scheme is a weak solution of the equation
\begin{equation}\label{equation-limit}
\ddot{u} - {\rm div}(A(t)\nabla u)=f\qquad \hbox{in} \ \Omega\,
\end{equation}
where $A(t)$ is the $G$-limit of $\sigma_{D_n(t)}$. 

Precisely we show the following.

\begin{proposition}
	Let $u$ be the weak limit, up to a subsequence, of $u_n$ as defined in \eqref{disc1}, and let $A(t)$ be the corresponding $G$-limit of $\sigma_{D_n(t)}$, then
\begin{equation}\label{weak-equation-limit}
-\int_0^T\int_\Omega\dot{u}\dot{\phi}\,dx\,dt +\int_0^T\int_\Omega A \nabla \tilde u\nabla\phi\,dx\,dt=\int_0^T\langle f,\phi\rangle\,dt
\end{equation}
for every $\phi\in H^1_0(0,T;H_0^1(\Omega))$, 
\end{proposition}

\begin{proof}
We start testing  \eqref{E-L-discrete} with $\phi\in H^1(0,T;H_0^1(\Omega))$, integrating in time. We obtain
	\begin{equation}\label{E-L-discrete-time}
-\int_0^T\int_\Omega v_n\dot{\phi}\,dx +\int_0^T\int_\Omega\sigma_{D_n} \nabla \tilde u_n\nabla\phi\,dx\, dt
=\int_0^T\langle f_n,\phi\rangle\,dt.
\end{equation}
Then taking the limit as $n\to+\infty$ we get
\begin{equation}\label{E-L-1}
-\int_0^T\int_\Omega \dot{u}\dot{\phi}\,dx +\int_0^T\int_\Omega\xi\nabla\phi\,dx\, dt
=\int_0^T\langle f,\phi\rangle\,dt
\end{equation}
where  $\xi(t)$ denotes  the weak limit in $L^2(\Omega)$ of $\sigma_{D_n(t)}\nabla \tilde u_n(t)$.
We only need to show that $\xi(t)=A(t)\nabla u(t)$ a.e. $t\in(0,T)$.

Let $\theta(t)$ be the weak star limit in $L^\infty(\Omega)$ of $\chi_{D_n(t)}$. By the monotonicity in $t$ of the damage sets $D_n(t)$ we deduce that $\Theta(t)=\int_{\Omega}\theta(t)dt$ is increasing and hence continuous up to a countable set of points in $(0,T)$.

Let $\tau\in(0,T)$ be a point of continuity of $\Theta(t)$ and $h>0$ and fix a test function $\phi\in H_0^1(\Omega)$ in \eqref{E-L-discrete}. Integrating in time from $\tau-h$ to $\tau$ we get
\begin{equation}\label{mean-equation}
\begin{split}
 \int_\Omega\sigma_{D_n(\tau)}& \left(\mean_{\tau-h}^\tau\nabla \tilde u_n(t)\,dt\right)\nabla\phi\,dx=-\int_\Omega\frac{v_n(\tau)-v_n(\tau-h)}{h}\phi\,dx\\
 &+ \int_\Omega \left(\mean_{\tau-h}^\tau(\sigma_{D_n(\tau)}-\sigma_{D_n(t)})\nabla \tilde u_n(t)\,dt\right)\nabla\phi\,dx+\langle \bar{f}_n,\phi\rangle\,,
\end{split}
\end{equation}
where $\bar{f}_n := -\!\!\!\!\!\!\int_{\tau-h}^\tau {f}_n(t)\,dt$ is the time average of $f_n$ in the interval $(\tau-h, \tau)$.
We also define
$$
\bar{u}_n=\mean_{\tau-h}^\tau \tilde u_n(t)\,dt\qquad \bar{u}=\mean_{\tau-h}^\tau u(t)\,dt
$$
and we denote by $\hat{u}_n$ the unique solution in $H^1_0(\Omega)$ of the elliptic equation
\begin{equation}\label{elliptic}
-{\rm div}(\sigma_{D_n(\tau)}\nabla\hat{u}_n)=-{\rm div}(A(\tau)\nabla\bar{u}).
\end{equation}
As a consequence of the $G$-convergence of $\sigma_{D_n(\tau)}$ to $A(\tau)$ we deduce that $\hat{u}_n$ converges to $\bar{u}$ weakly in $H^1_0(\Omega)$ and hence that $\hat{u}_n-\bar{u}_n\weakly 0$ weakly in $H^1_0(\Omega)$, and $\sigma_{D_n(\tau)}\nabla\hat{u}_n$ weakly converges to $A(\tau)\nabla\bar{u}$ in $L^2$. Using $\hat{u}_n-\bar{u}_n$ as test function in \eqref{mean-equation} and \eqref{elliptic} we get
\begin{equation*}%
\begin{split}
\int_\Omega\sigma_{D_n(\tau)}|\nabla(\hat{u}_n-\bar{u}_n)|^2dx=&-\int_\Omega\frac{v_n(\tau)-v_n(\tau-h)}{h}(\hat{u}_n-\bar{u}_n)\,dx\cr-& \int_\Omega \left(\mean_{\tau-h}^\tau(\beta-\alpha)\chi_{D_n(\tau)\setminus D_n(t)}\nabla \tilde u_n(t)\,dt\right)\nabla(\hat{u}_n-\bar{u}_n)\,dx\\
+&\langle \bar f_n,\hat{u}_n-\bar{u}_n\rangle 
-\int_\Omega (A(\tau)\nabla\bar u)\nabla(\hat{u}_n-\bar{u}_n)\,dx.
\end{split}
\end{equation*}
From this, using the boundness of $v_n$ in $W^{1,\infty}(0,T; L^2(\Omega))$, the boundness of $\bar f_n$ in $L^2(\Omega)$, and the strong $L^2$ convengence to zero of $\hat{u}_n-\bar{u}_n$ we get
\begin{equation}
\begin{split}\label{estimate1}
&\int_\Omega\sigma_{D_n(\tau)}|\nabla(\hat{u}_n-\bar{u}_n)|^2dx\\
\leq&  \int_\Omega \left(\mean_{\tau-h}^\tau(\beta-\alpha)\chi_{D_n(\tau)\setminus D_n(t)}|\nabla \tilde u_n(t)|\,dt\right)|\nabla(\hat{u}_n-\bar{u}_n)|\,dx + o(1)\cr
\leq& C \left[\int_\Omega \left(\mean_{\tau-h}^\tau\chi_{D_n(\tau)\setminus D_n(t)}|\nabla \tilde u_n(t)|\,dt\right)^2\,dx\right]^{\frac12} \|\nabla(\hat{u}_n-\bar{u}_n)\|_{L^2}+ o(1)
\end{split}
\end{equation}
where we also applied H\"older inequality. Now by Jensen inequality and Remark 1 and \eqref{weak-threshold}, we get
\begin{equation}\label{estimate2}
\begin{split}
\int_\Omega  \big(\mean_{\tau-h}^\tau\chi_{D_n(\tau)\setminus D_n(t)}|&\nabla \tilde u_n(t)|\,dt\big)^2\,dx\\
\leq &\int_\Omega\mean_{\tau-h}^\tau\chi_{D_n(\tau)\setminus D_n(t)}|\nabla \tilde u_n(t)|^2\,dt\,dx \cr
= &\mean_{\tau-h}^\tau\int_{D_n(\tau)\setminus D_n(t)}|\nabla \tilde u_n(t)|^2\,dx\,dt\cr
\leq & M^2\mean_{\tau-h}^\tau|D_n(\tau)\setminus D_n(t)|\, dt +o(1)\\ \leq &
M^2|D_n(\tau)\setminus D_n(\tau-h)|+o(1)\,.
\end{split}
\end{equation}
Applying Young's inequality from \eqref{estimate1} and \eqref{estimate2} we obtain that there exists a constant $C>0$ such that
\begin{equation}\label{estimate3}
\int_\Omega|\nabla(\hat{u}_n-\bar{u}_n)|^2dx\leq C |D_n(\tau)\setminus D_n(\tau-h)| + o(1)\,,
\end{equation}
and therefore
\begin{equation}\label{estimate4}
\limsup_{n\to \infty}\int_\Omega|\nabla(\hat{u}_n-\bar{u}_n)|^2dx\leq C (\Theta(\tau)- \Theta(\tau-h))\,.
\end{equation}
From this we get
\begin{equation*}
\begin{split}
\int_\Omega\big|\mean_{\tau-h}^\tau\sigma_{D_n(t)}\nabla\tilde u_n(t)\,dt -&\sigma_{D_n(\tau)}\nabla\hat u_n  \big|^2dx\\\leq & 2\int_\Omega\left(\mean_{\tau-h}^\tau(\beta-\alpha)\chi_{D_n(\tau)\setminus D_n(t)}|\nabla\tilde u_n(t)|\,dt \right)^2dx\cr
+ & 2\beta^2\int_{\Omega}|\nabla(\hat{u}_n-\bar{u}_n)|^2dx\cr \leq & C (\Theta(\tau)- \Theta(\tau-h)) +o(1)\,.
\end{split}
\end{equation*}
Now by the definition of $\hat u_n$, taking the limit as $n\to+\infty$ we get
\begin{equation}\label{last-estimate}
\int_\Omega\left|\mean_{\tau-h}^\tau\xi(\tau)\,dt -A(\tau)\nabla\bar u  \right|^2dx\leq C (\Theta(\tau)- \Theta(\tau-h))\,.
\end{equation}
Using the fact that a.e. $\tau\in(0,T)$ is a Lebesgue point of $\sigma(\tau)$ and $u(\tau)$ and a continuity point for $\Theta(\tau)$, taking the limit as $h\to 0$ we get
\begin{equation}\label{very-last-estimate}
\int_\Omega|\xi(\tau) -A(\tau)\nabla u(\tau) |^2dx\leq \lim_{h\to 0} C (\Theta(\tau)- \Theta(\tau-h))=0\,,
\end{equation}
which concludes the proof.

\end{proof}

It remains to prove the energy inequality (\ref{enbal}). This is done in the following lemma where we use the same technique as in \cite{DSc}.
We define
$$E_{tot}(t,u,\theta,A):=\frac{1}{2}\normm{\dot{u}}_{L^2}^2+\frac{1}{2}\int_{\Omega}A\nabla u \nabla u dx+k\int_{\Omega}\theta dx-\langle f(t),u\rangle$$
\begin{lemma}
Let the triple $(u(t),\theta(t),A(t))$ be the limit of  $(u_n(t),\chi_{D_n}(t),\sigma_{D_n}(t))$ and denote
	$$E_{tot}(t):=E_{tot}(t,u(t),\theta(t),A(t)).$$
	It holds
	$$E_{tot}(t)\leq E_{tot}(0^+)-\int_{0}^{t}\langle \dot{f}(s),u(s)\rangle ds.$$
\end{lemma}

\begin{proof}
	Using the  almost minimality condition of $(\U{i+1},\Dn{i+1})$, (\ref{mind}), we have
	\begin{equation}\label{minim}
	\begin{split}
	&  \frac{1}{2}\int_{\Omega}\sigma_{\Dn{i+1}}|\nabla \U{i+1}|^2dx+k|\Dn{i+1}|  +\frac{1}{2}  \left \|\frac{\U{i+1}-\U{i}}{\Delta t}
 -\frac{\U{i}-\U{i-1}}{\Delta t}\right \|_{L^2}^2		\\
 &\qquad-\langle f^{n}_{i+1},u^{n}_{i+1}\rangle 
	\\
	 &\leq \frac{1}{2}\int_{\Omega}\sigma_{D}|\nabla u|^2dx+k|D|+\frac{1}{2}\left \|\frac{u-\U{i}}{\Delta t}-\frac{\U{i}-\U{i-1}}{\Delta t}\right \|_{L^2}^2
	 \\
	&\qquad -\langle f^{n}_{i+1},u\rangle+\frac{\tau^2_n}{2^i },
	\end{split}
	\end{equation}
	for $D\supseteq \Dn{i}$ and $u\in H^{1}_0(\Omega)$.\\
	Now for any $\delta \in (0,1)$ we choose a set $E_\delta$ with the following properties
	\begin{itemize}
		\item $E_{\delta}\subseteq \Dn{i+1}\setminus \Dn{i}$,
		\item $|E_{\delta}|=\delta(|\Dn{i+1}|-|\Dn{i}|)$.
	\end{itemize}
	We then take a test set in \eqref{minim} given by  $D=\Dn{i+1}\backslash{E_{\delta}}$, which then satisfies
	$$|D|=\delta(|\Dn{i}|-|\Dn{i+1}|)+|\Dn{i+1}|, \qquad \text{and} \qquad \sigma_{D}=\sigma_{\Dn{i+1}}+(\beta-\alpha)\chi_{E_{\delta}}.$$
	Therefore the right hand side of (\ref{minim}) becomes
	\begin{equation}\label{rig}
	\begin{split}
	& \frac{1}{2}\int_{\Omega}\sigma_{\Dn{i+1}}|\nabla u|^2dx+\frac{(\beta-\alpha)}{2}\int_{E_{\delta}}|\nabla u|^2dx +k\delta (|\Dn{i}|-|\Dn{i+1}|)\cr & + k|\Dn{i+1}|+\frac{1}{2}  \left \|\frac{u-\U{i}}{\Delta t}-\frac{\U{i}-\U{i-1}}{\Delta t}\right \|_{L^2}^2-\langle f^{n}_{i+1},u\rangle + \frac{\tau^2_n}{2^i}.
	\end{split}
	\end{equation}
	Now we consider as test function $\bar{u}:=\U{i+1}-\delta(\U{i+1}-\U{i})$, then
	\begin{equation}\label{one}
	\begin{split}
	&\frac{1}{2}\int_{\Omega}\sigma_{\Dn{i+1}}|\nabla\bar{u}|^{2}dx-\langle f^{n}_{i+1},\bar{u}\rangle=\frac{1}{2}\int_{\Omega}\sigma_{\Dn{i+1}} |\nabla\U{i+1}|^{2}dx\cr &+ \delta\bigg(\frac{\delta-2}{2} \bigg)\int_{\Omega}\sigma_{\Dn{i+1}}|\nabla\U{i+1}|^{2}dx  +\frac{\delta^{2}}{2}\int_{\Omega}\sigma_{\Dn{i+1}}|  \nabla \U{i}|^{2}dx\cr &+ \delta(1-\delta)\int_{\Omega}\sigma_{\Dn{i+1}}\nabla{\U{i+1}}\nabla\U{i}dx-\langle f^{n}_{i+1},u^{n}_{i+1}\rangle+\delta\langle f^{n}_{i+1},u^{n}_{i+1}-u^n_i\rangle
	\end{split}
	\end{equation}
	and
	\begin{equation}\label{second}
	\begin{split}
	\frac{1}{2}  \Big\|\frac{\bar{u}-\U{i}}{\Delta t}&-\frac{\U{i}-\U{i-1}}{\Delta t}\Big\|_{L^2}^2\cr & =\frac{1}{2}  \left \|\frac{\U{i+1}-\U{i}}{\Delta t}-\frac{\U{i}-\U{i-1}}{\Delta t}\right \|_{L^2}^2 \cr & +  \frac{\delta^{2}}{2} \left \|\frac{\U{i+1}-\U{i}}{\Delta t}\right \|_{L^2}^2 \cr &-\delta\int_{\Omega}\bigg(\frac{\U{i+1}-\U{i}}{\Delta t} -\frac{\U{i}-\U{i-1}}{\Delta t}\bigg)\bigg(\frac{\U{i+1}-\U{i}}{\Delta t}\bigg)dx.
	\end{split}
	\end{equation}
	
	Combining (\ref{minim}), (\ref{rig}), (\ref{one}), (\ref{second}) and dividing by $\delta$ we have
	\begin{equation}\label{uff}
	\begin{split}
	\bigg(\frac{2-\delta}{2}\bigg)\int_{\Omega}&\sigma_{\Dn{i+1}}|\nabla \U{i+1}|^{2}dx   -(1-\delta)\int_{\Omega}\sigma_{\Dn{i+1}}\nabla\U{i}\nabla\U{i+1}dx  \cr &  + \int_{\Omega}\bigg(\frac{\U{i+1}-\U{i}}{\Delta t} -\frac{\U{i}-\U{i-1}}{\Delta t} \bigg)(\frac{\U{i+1}-\U{i}}{\Delta t}) dx
	\cr  \leq\frac{\delta}{2}\int_{\Omega}\sigma_{\Dn{i+1}}&|\nabla\U{i}| ^{2}dx+\frac{\delta}{2}\left \|\frac{\U{i+1}-\U{i}}{\Delta t}\right \|^{2} +k(|\Dn{i}|- |\Dn{i+1}|)\cr & +\frac{(\beta-\alpha)}{2}\int_{E_{\delta}}|\nabla\bar{u}|^{2}dx \cr & + \Delta t\langle f^{n}_{i+1},\frac{u^{n}_{i+1}-u^n_i}{\Delta t}\rangle +\frac{\tau_n^2}{\delta 2^i }.
	\end{split}
	\end{equation}
	We note that the first and third term of the left handside satisfy
	\begin{equation}\label{re}
	\begin{split}
	\bigg(\frac{2-\delta}{2}\bigg)\int_{\Omega}\sigma_{\Dn{i+1}}|\nabla \U{i+1}|^{2}dx & -(1-\delta)\int_{\Omega}\sigma_{\Dn{i+1}}\nabla\U{i}\nabla\U{i+1}dx \cr & \geq (1-\delta)\int_{\Omega}\sigma_{\Dn{i+1}}\nabla{\U{i+1}}\nabla(\U{i+1}-\U{i})dx.
	\end{split}
	\end{equation}
	Now considering the following identities (see \cite{DSc} pages 14,15, and 16):
\begin{equation}
\begin{split}
&\int_{\Omega}\sigma_{\Dn{i+1}}\nabla\U{i+1}(\nabla\U{i+1}-\nabla\U{i})dx\\
&=
 \int_{t_i}^{t_{i+1}}\int_{\Omega}\sigma_{\Dn{i+1}}\nabla u_n(s)\nabla\dot{u}_n(s)ds+\frac{\Delta t}{2}\int_{t_i}^{t_{i+1}}\int_{\Omega}\sigma_{\Dn{i+1}}\nabla\dot{u}_n(s)\nabla\dot{u}_n(s)ds
 \end{split}
\end{equation} 
\begin{equation}
\begin{split}
&\int_{\Omega}\bigg(\frac{\U{i+1}-\U{i}}{\Delta t} -\frac{\U{i}-\U{i-1}}{\Delta t}\bigg)(\frac{\U{i+1}-\U{i}}{\Delta t})\\&=\frac{1}{2}\left\|\frac{\U{i+1}-\U{i}}{\Delta t} \right\|_{L^2}^{2}
-\frac{1}{2}\left\|\frac{\U{i}-\U{i-1}}{\Delta t} \right\|_{L^2}^{2}+\frac{\Delta t}{2}\int_{t_{i}}^{t_{i+1}}\|\dot{v}_n(s)\|^{2}ds
\end{split}
\end{equation}
		
		\begin{equation}
		\begin{split}
		\frac{\delta}{2}\int_{\Omega}\sigma_{\Dn{i+1}}|\nabla \U{i}|^2dx=\frac{\delta}{2\Delta t}\int_{t_i}^{t_{i+1}}\int_{\Omega}\sigma_{\Dn{i+1}}|\nabla u_{n}(s)|^{2}ds\\
		-\frac{\delta}{2}\int_{t_i}^{t_{i+1}}\int_{\Omega}\sigma_{\Dn{i+1}}\nabla u_{n}(s)\nabla\dot{u}(s)ds\\+\frac{\delta\Delta t}{12}\int_{t_i}^{t_{i+1}}\int_{\Omega}\sigma_{\Dn{i+1}}|\nabla \dot{u}_{n}(s)|^2ds
		\end{split}
		\end{equation}
	and using (\ref{re}) the inequality (\ref{uff}) becomes
	\begin{equation*}
	\begin{split}
	&(1-\delta)\int_{t_i}^{t_{i+1}}  \int_{\Omega}\sigma_{\Dn{i+1}}\nabla u_n(s)  \nabla\dot{u}_n(s)ds\\&+(1-\delta)\frac{\Delta t}{2}\int_{t_i}^{t_{i+1}}\int_{\Omega}\sigma_{\Dn{i+1}}\nabla\dot{u}_n(s)\nabla\dot{u}_n(s)ds \cr &+\frac{1}{2}\left\|\frac{\U{i+1}-\U{i}}{\Delta t} \right\|_{L^2}^{2}-\frac{1}{2}\left\|\frac{\U{i}-\U{i-1}}{\Delta t} \right\|_{L^2}^{2}+\frac{\Delta t}{2}\int_{t_{i}}^{t_{i+1}}|\dot{v}_n(s)|^{2}ds \cr &
	\leq \frac{\delta}{2\Delta t}\int_{t_i}^{t_{i+1}}\int_{\Omega}\sigma_{\Dn{i+1}}|\nabla u_{n}(s)|^{2}ds-\frac{\delta}{2}\int_{t_i}^{t_{i+1}}\int_{\Omega}\sigma_{\Dn{i+1}}\nabla u_{n}(s)\nabla\dot{u}(s)ds \cr & +\frac{\delta\Delta t}{12}\int_{t_i}^{t_{i+1}}\int_{\Omega}\sigma_{\Dn{i+1}}|\nabla \dot{u}_{n}(s)|^2ds+\frac{\delta}{2\Delta t} \int_{t_i}^{t_{i+1}}\normm{\dot{u_n}(s)}ds \cr &+k(|\Dn{i}|-|\Dn{i+1}|)+(\beta-\alpha)\int_{E_{\delta}}|\nabla\bar{u}|^2dx+\int_{t_{i}}^{t^{i+1}}\langle f^n_{i+1},\dot{u}_n(s)\rangle ds+\frac{\tau_n^2}{\delta 2^i }
	\end{split}
	\end{equation*}
	\text{}\\
	which can be rewritten in the form
	\begin{equation*}
	\begin{split}
	&\frac{2-\delta}{2}\int_{t_i}^{t_{i+1}}  \int_{\Omega}\sigma_{\Dn{i+1}}\nabla u_n(s)  \nabla\dot{u}_n(s)ds +k(|\Dn{i+1}|-|\Dn{i}|) \cr &+
	\frac{1}{2}\left\|\frac{\U{i+1}-\U{i}}{\Delta t} \right\|_{L^2}^{2}-\frac{1}{2}\left\|\frac{\U{i}-\U{i-1}}{\Delta t} \right\|_{L^2}^{2}+\frac{\Delta t}{2}\int_{t_{i}}^{t_{i+1}}|\dot{v}_n(s)|^{2}ds \cr & \leq
	\frac{(7\delta-6)\Delta t}{12}\int_{t_i}^{t_{i+1}}\int_{\Omega}\sigma_{\Dn{i+1}}|\nabla \dot{u}_{n}(s)|^2ds  +\frac{\delta}{2\Delta t} \int_{t_i}^{t_{i+1}}\normm{\dot{u}_n(s)}ds \cr & +   \frac{\delta}{2\Delta t}\int_{t_i}^{t_{i+1}}\int_{\Omega}\sigma_{\Dn{i+1}}|\nabla u_{n}(s)|^{2}ds\\ &+  (\beta-\alpha)\int_{E_{\delta}}|\nabla\bar{u}|^2dx-\int_{t_{i}}^{t^{i+1}}\langle \dot{f}(s),{u}_n(s)\rangle ds+\frac{\tau_n^2}{\delta 2^i }
	\end{split}
	\end{equation*}
	Moreover let note that  the first term on the left hand side of the latter inequality can be rewritten (integrating by part) as 
	\begin{equation*} \label{fir}
	\frac{2-\delta}{4}\bigg[\int_{\Omega}\!\sigma_{\Dn{i+1}}|\nabla u_n(t_{i+1})|^2dx-\int_{\Omega}\!\sigma_{\Dn{i}}|\nabla u_n(t_{i})|^2dx+(\beta-\alpha)\!\int_{\Dn{i+1}\backslash \Dn{i}}\!\!\!\!|\nabla u_n(t_i)|^2dx\bigg]
	\end{equation*}
	with
	\begin{equation}\label{mag}
	(\beta-\alpha)\int_{\Dn{i+1}\backslash \Dn{i}}|\nabla u_n(t_i)|^2dx\geq 0.
	\end{equation}
	Therefore summing over $i=1,\ldots, j$, with $t\in (t_{j-1},t_j)$, we obtain
	\begin{equation*}
	\begin{split}
	&\frac12\int_{\Omega}\sigma_{D_n(t)}|\nabla u_n(t)|^2dx-\frac12\int_{\Omega}\sigma_{D_n(\tau_n)}|\nabla u_n(\tau_n)|^2dx + K(|D_n(t)|-|D_n(\tau_n)|)\\
	&
	\frac12\|v_n(t)\|^2_2-\frac12\|v_n(\tau_n)\|^2_2 + \frac{\tau_n}{2}\int_{\tau_n}^t\|v_n(s)\|^2_2 ds\\
	&\qquad \leq -\int_{\tau_n}^t\langle \dot{f}(s), u_n(s)\rangle ds + C\frac{\delta}{\tau_n} +o(\delta) +\frac{\tau_n^2}{\delta}.
	\end{split}
	\end{equation*}
	We then conclude choosing $\delta=\delta_n\to 0$ as $n\to \infty$, with $\tau_n^2<\!\!<\delta_n<\!\!<\tau_n$, so that $\frac{\delta_n}{\tau_n}\to 0$ and $\frac{\tau_n^2}{\delta}\to 0$, and using the convergence of $u_n$, $\chi_{D_n}$ and $\sigma_{D_{n}}$.
\end{proof}

\section{ Threshold condition}\label{sec-threshold}
We now prove the threshold property as stated in (\ref{thrdisc}) using the blow-up argument proposed in \cite{GL}.
\begin{proposition}
Given $(u_n(t), D_n(t))$ as in (\ref{disc1}) we have the following threshold condition
\begin{equation}\label{thco}
\lim_{n\lra\infty}|\{x\notin D_n(t) : |\nabla u_n(t)|>\lambda+\delta\}|=0
\end{equation}
for each $\delta>0$ and $\lambda:=\sqrt{\frac{2\alpha k}{\beta(\beta-\alpha)}}$.
\end{proposition}
\begin{proof}
We first prove the result for $(u^n_i,D^n_i)$ then by a convexity argument we will easily obtain the claim.
The first part of the proof is similar to the one in \cite{GL}.\\
Given a set $Q\subseteq\Omega$ we define
$$E(u,D,Q):=\frac{1}{2}\int_{Q}\sigma_{D\cap Q}|\nabla u|+k|D\cap Q|.$$
We set
$$E^n:=E^{n}_{i,\delta}:=\{x\notin D_i^n : |\nabla u^n_i(x)|>\lambda+\delta\}$$
and we suppose by contradiction that there exists $\delta>0$ such that 
$$\limsup_{n\lra\infty}|E^{n}_{i,\delta}|=2\eta$$
with $\eta>0$, which implies that (up to subsequences) 
\begin{equation}\label{et}
|E^{n}|>\eta
\end{equation}
for $n\geq \bar{n}$ for a fixed $\bar{n}>\!\!>1$.\\
We first show that for $n>\!\!>1$ (so using (\ref{et})) there exists an explicit constant $c>0$ and $(\tilde{w}^n_i,\tilde{D}^n_i)$ admissible for the minimum problem (\ref{incremental}) such that
\begin{equation}\label{minel}
E(\tilde{w}^n_i,\tilde{D}_i^n,\Omega)\leq E(u^n_i,{D}_i^n,\Omega)-c,
\end{equation} 
i.e., decreasing the elastic part of the energy in the whole $\Omega$. We do this also showing that the change in the kinetic part of the energy can be made arbitrarily small, so that we obtain competitors with total energy less than that of $(u^n_i,D^n_i)$, which is a contradiction because of the minimality property of $(u^n_i,D^n_i)$. \\

To begin, for each $n\geq \bar{n}$ and $\varepsilon>0$ we consider a covering of $E^n$ made of squares $Q$  such that 

\begin{itemize}
\item[1)] the center $\bar{x}$ of the square is in $E_n$ and it is a Lebesgue point for $u^n_i$ and $\nabla {u^n_i}$, i.e., it holds
$$\lim_{r\lra 0^{+}}\frac{1}{|B_{r}(\bar{x})|}\int_{B_{r}(\bar{x})}|u^n_i(\bar{x})-u^n_i(y)|^{p}dy=0$$
(and the same for $\nabla{u^n_i}$) for all $p\geq 1$, where $B_{r}(\bar{x})$ is the ball with center $\bar{x}$ and radius $r$.
\item[2)] two sides of $Q$ are orthogonal to $\nabla u^n_i(\bar{x});$
\item[3)] defined $\bar{u}^n_i(x):=u^n_i(\bar{x})+\nabla u^n_i(\bar{x})\cdot    (x-\bar{x})$ we have 
$$\normm{u^n_i-\bar{u}^n_i}_{H^1(Q)}^{2}\leq\varepsilon |{Q}|$$
$$|{D_i^n}\cap {Q}|\leq\varepsilon|Q|.$$
\end{itemize}
Note that since $Q$ is a square of a covering of $E^n$ it depends on $n,\delta$ and $i$ and by definition it depends also on $\varepsilon$ and in general its measure goes to zero when $\varepsilon$ goes to zero. Moreover for each $\varepsilon$ it is a fine convering of $E^n$ so we can choose a finite number of disjoint square to cover $E^n$, except for a set of measure less than $\varepsilon$.

We divide the proof into 3 steps which we first sketch and then detail.
For the \emph{first step} we will show that considering test functions in each ${Q}$ with the same boundary condition of $\bar{u}^n_i$ in $\partial Q$ (instead of ${u}^n_i$) we can decrease in this square the elastic energy given by the pair $({\bar{u}}^n_i,\emptyset)$  using a process of lamination, in particular we will show that for each $\sigma>0$ there exist $v^n_i$ and ${{\hat{D}}_i^n}$ such that
\begin{equation} \label{minvatd}
E(v^n_i,{\hat{D}}_i^n,{Q})\leq E(\bar{u}^n_i,\emptyset,{Q})-\frac{1}{4}\beta\delta^2|{Q}|
\end{equation}
with $v^n_i=\bar{u}^n_i$ on $\partial Q$ and $\|v_i^n-\bar{u}^n_i\|_{L^2(Q)}<\sigma|Q|$.

To do it we recall a technical result to match the boundary conditions of special (almost) test functions (which will be piecewise linear functions) with the boundary conditions of $\bar{u}^n_i$.\\
In the \emph{second step}, using the previous one, we will show that in each square ${Q}$ we can lower the energy given by $(u^n_i,D_i^n)$ using a test function with the same boundary conditions of $u^n_i$ and choosing $\varepsilon$ sufficently small, i.e., we will show that for each $s>0$  we can choose $\varepsilon$ small in such a way that there exists $\hat{w}^n_{i}$ and $\hat{D}_n^{i}\subset Q$   such that
\begin{equation} \label{minvat2}
E(\hat{w}^n_{i},\hat{D}_n^{i},{Q})\leq E({u}^n_i,D_i^n,{Q})-\frac{1}{4}\beta\delta^2|{Q}| +s|Q|
\end{equation}
with $\hat{w}^n_{i}={u}^n_i$ on $\partial Q$ and $\|\hat{w}^n_{i}-{u}^n_i\|_{L^2(Q)}<\sigma |Q|$.\\
Finally in the \emph{third step} we will use the previous steps to construct an admissible pair for the problem (\ref{incremental}) that has in $\Omega$ energy lower than the one given by ($u^n_i,D_i^n$).\\
\smallskip

\emph{Step 1}. We consider an arbitrary square of the (almost) covering of $E^n$. 
To avoid heavy notation we can assume that $\bar{x}=0$ and $u^n_i(\bar{x})=0$ and so we have $\bar{u}^n_i(x)=\nabla{u}^n_i(0)\cdot    x$\\
We consider the continuous periodic function $z(y)$ such that $z(0)=0$, $z(1)=|\nabla u^n_i(0)|=\lambda+\tilde\delta$ such that
\begin{equation}
z'(y)=\begin{cases} \frac{\beta}{\alpha}\lambda, & \mbox{if }y\in (0,d) \\
 \lambda, & \mbox{if }y\in[d,1)
\end{cases}
\end{equation}
where $d$ is given by $d=\frac{\tilde\delta\alpha}{\lambda(\beta-\alpha)}$ and we define
$${\bar{v}}_{i,h}^n(x):=h z\bigg(\frac{x}{h}\cdot   \frac{\nabla{u^n_i(0)}}{|\nabla u^n_i(0)|}\bigg)$$
$$\hat{D}^{n}_{i,h}:=\bigg\{ x\in Q : z'\big(\frac{x}{h}\cdot   \frac{\nabla{u^n_i(0)}}{|\nabla u^n_i(0)|}\big)=\frac{\beta}{\alpha}\lambda \bigg\}.$$

Note that by definition
$$
|\nabla {\bar{v}}_{i,h}^n(x)|= \left|z'\left(\frac{x}{h}\cdot   \frac{\nabla{u^n_i(0)}}{|\nabla u^n_i(0)|}\right)\right|.
$$

By the periodicity of $z$ the sequence $\nabla{\bar{v}}^n_{i,h}$ converges weakly in $L^2(Q)$ to $|\nabla u^n_i(0)|$ when $h\lra 0$, and then ${\bar{v}}^n_{i,h}$ converges  strongly  in $L^2({Q})$, to $\bar{u}^n_i$  and it is bounded in $H^1(Q)$. Now we match the boundary conditions of ${\bar{v}}^n_{i,h}$ with the ones of $\bar{u}^n_i$ using the cut-off function 
\begin{equation}
\phi(y)=\begin{cases} 1, & \mbox{if }y\in Q\backslash Q_R \\
 0, & \mbox{if }y\in Q_{R-\mu}
\end{cases}
\end{equation}
with $|\nabla\phi|=\frac{1}{\mu}$ in $Q_{R}\backslash Q_{R-\mu}$, where $R\in (0,{|{{Q}}|}^{1/N})$, and $\mu\in (0,R)$. We define $$v^n_{i,h}=\phi \bar{u}^n_i+(1-\phi){\bar{v}}_{i,h}^n.$$
Note that we can choose $R$ such that 
\begin{equation}\label{ap1}
\lim_{\mu\lra0}\lim_{h\lra 0}\int_{Q_{R}\backslash Q_{R-\mu}}|\nabla v^n_{i,h}|^2=0
\end{equation}
and
\begin{equation}\label{ap2}
\lim_{\mu\lra 0}\lim_{h\lra 0}\int_{Q_{R}}|\nabla v^n_{i,h}-\nabla {\bar{v}}^n_{i,h}|^2=0
\end{equation}
(cfr. details see \cite{GL}, Remark 13).
By  (\ref{ap1}) and the (\ref{ap2}), it follows that for $n$ large enough, the  energy in $|Q|$ given by $(v^n_{i,h},\hat{D}^n_{i,h})$  is arbitrarily close to that given by $(\bar{v}^n_{i,h},{\hat{D}}^n_{i,h})$ which is (with a simple computation)
$$
\frac{1}{2}\beta\lambda|\nabla u^n_i(0)||{Q}|+\frac{1}{2}\beta\lambda{\tilde\delta}|{Q}|.
$$

So, we conclude that for each $\sigma>0$ there exists $v^n_i:=v^n_{i,h}$ and $\hat{D}^n_{i}:=\hat{D}^n_{i,h}$ (with $h<\!\!<1$) such that
$\|v_i^n-\bar{u}_i^n\|_{L^2(Q)}<\frac\sigma2 |Q|$ and
\begin{equation}\label{vni}
\begin{split}
E(v^n_{i},\hat{D}^n_{i},Q)&\leq E(\bar{v}^n_{i,h},\hat{D}^n_{i,h},Q)+ \sigma|Q|\\& = E(\bar{u}^n_i,\emptyset,{Q})-\frac{1}{2} \beta( |\nabla u^n_i(0)  |-\lambda)^2|{Q}|+\sigma |Q| \\&<E(\bar{u}^n_i,\emptyset,{Q})-\frac{1}{4}\beta\delta^2 |{Q}|,
\end{split}
\end{equation}
where the last inequality comes from the fact that  $\tilde\delta=(|\nabla u^n_{i}(0)|-\lambda)>\delta$ and so (\ref{minvatd}).\\
\smallskip

\emph{Step 2}. We start showing that by the properties (3) of ${Q}$ we have 
\begin{equation}\label{uf}
|E(u^n_i,D_i^n,Q)-E(\bar{u}^n_i,\emptyset,{Q})|\leq o_{\varepsilon}(1)|{Q}|
\end{equation}
indeed we have
\begin{equation*}
\begin{split}
E(u^n_i,D_i^n,Q)-E(\bar{u}^n_i,\emptyset,{Q})&=\int_{Q}\sigma_{D_i^n}|\nabla u^n_i|^{2}dx+k|D_i^n\cap Q|-\beta\int_{Q}|\nabla \bar{u}^n_i|^2dx\\
&\leq \beta\int_{{Q}}(|\nabla u^n_i|^2-|\nabla{\bar{u}^n_i}|^2)dx+k|D_i^n\cap {Q}|
\end{split}
\end{equation*}
and using the property for numbers $|a|^2-|b|^2\leq |a-b|^{2}+2|a-b||b|$ (and also the Holder inequality) we obtain
\begin{equation*}
E(u^n_i,D_i^n,Q)-E(\bar{u}^n_i,\emptyset,{Q})\leq C(\|\nabla u^n_i-\nabla{\bar{u}^n_i}\|^2_{L^2(Q)}+\|\nabla u^n_i-\nabla{\bar{u}^n_i}\|_{L^2(Q)})+k|D_i^n\cap {Q}|,
\end{equation*}
with $C>0$. The opposite inequality follows similarly. Then using the properties (3) we have (\ref{uf}). \\
Now  the function $\hat{w}^n_i:=v^n_i+(u^n_i-\bar{u}^n_i$), with $v^n_i$ as in (\ref{vni}), has the same boundary condition in $Q$ of $u^n_i$. So by (\ref{minvatd}) and (\ref{uf}) we have that for each $s>0$ we can take $\varepsilon$ sufficiently small in such a way that 
\begin{equation} 
E(\hat{w}^n_{i},\hat{D}_n^{i},{Q})\leq E({u}^n_i,D_i^n,{Q})-\frac{1}{4}\beta\delta^2|{Q}| +s|Q|,
\end{equation} 
i.e., the inequality (\ref{minvat2}), as well as $\|\hat{w}_i^n-u_i^n\|_{L^2(Q)}<\sigma|Q|$.\\
\smallskip

\emph{Step 3}.
We come back to whole of $\Omega$ and consider the problem
$$\inf_{w,D'}\bigg\{\frac{1}{2}\int_{\Omega}\sigma_{D'}|\nabla{w}|^{2}dx+k|D'| : (w-u^n_i)\in H^{1}_0(\Omega), D'\supseteq D_i^n \bigg\}.$$
Iterating the previous steps for each square of the covering of $E^n$ we can construct for each $s>0$ and $\sigma>0$ a pair $(\tilde{w}^n_i,\tilde{D}_i^n)$ with $\tilde{D}_i^n\supseteq D_i^n$ and $\tilde{D}_i^n\setminus E^n=D_i^n\setminus E^n$, and $w^n_i=u^n_i$ outside $E_n$, such that

\begin{equation}\label{abs0}
E(\tilde{w}^n_i,\tilde{D}_i^n,\Omega)\leq E(u^n_i,D_i^n,\Omega)+(-\frac{1}{4}\beta\delta^2+s)|E^n|.
\end{equation}
Since for $n\geq \bar{n}$, $\varepsilon<\!\!<1$ and $\delta>0$ we can have $\sigma$ and $s$ small as we want we obtain that there exist $c>0$ such that
$$
E(\tilde{w}^n_i,\tilde{D}_i^n,\Omega)\leq E(u^n_i,{D}_i^n,\Omega)-c,
$$
i.e., the inequality (\ref{minel}), together with $\|\tilde{w}_i^n-u_i^n\|_{L^2(Q)}<\sigma c$.\\

Now we link the result for the elastic part and dissipation of the energy of $(u^n_i, D_i^n)$ with the (almost) minimality property of $(u^n_i, D_i^n)$ for the \textit{total} energy (elastic+kinetic) in such a way to obtain a contradition and so the validity of (\ref{thco}).
By the (almost) minimality property of $(u^n_i,D_i^n)$ we have 
\begin{equation}\label{ami}
\begin{split}
& E(u^n_i, D_i^n)  +\frac{1}{2} \left \| \frac{u^n_i-u_{i-1}^n}{\Delta t}- \frac{u_{i-1}^n-u_{i-2}^n}{\Delta t}   \right\|^{2}_{L^{2}(\Omega)}  \leq \\
& E(\tilde{w}^n_i, \tilde{D}_i^n)  +\frac{1}{2}   \left \| \frac{\tilde{w}^n_i-u_{i-1}^n}{\Delta t}-  \frac{u_{i-1}^n-u_{i-2}^n}{\Delta t}\right\|^{2}_{L^{2}(\Omega)} + \frac{\tau_n}{2^i }
\end{split}
\end{equation}
with $(\tilde{w}^n_i,\tilde{D}^n_i)$ as in (\ref{abs0}). But for $\sigma$ small, the two norms above are arbitrarily close.  So, for $n$ large, we get a contradiction and we conclude the proof. 
\end{proof}

\end{document}